\documentclass[letterpaper, 10 pt, conference]{ieeeconf}  

\IEEEoverridecommandlockouts                              

\usepackage{amsmath,amssymb,amsfonts, amsthm}
\usepackage{wrapfig}
\usepackage{subcaption}
\usepackage{multirow}
\usepackage{graphicx} 
\usepackage{subcaption}
\usepackage{mathrsfs}
\usepackage{hyperref}
\usepackage[noadjust]{cite}

\newcommand{\Sn}{\mathcal S_N}
\newcommand{\tr}{\mathrm{tr}}
\newcommand{\Ss}{\bar{\mathbb E}}

\newcommand{\checkW}{\mathcal{\check W}}
\usepackage{xcolor}

\newtheorem{theorem}{Theorem}
\newtheorem{lemma}[theorem]{Lemma}

\newtheorem{assumption}{Assumption}
\newtheorem{remark}{Remark}
\newtheorem{definition}{Definition}

\title{\LARGE \bf
Regret Analysis with Almost Sure Convergence for OBF-ARX Filter
}

\author{Jiayun Li$^{1}$, Yiwen Lu$^{1}$ and Yilin Mo$^{1}$
\thanks{$^{1}$Jiayun Li, Yiwen Lu and Yilin Mo are with the Department of Automation and BNRist, Tsinghua University, Beijing 100084, P.R.China. Emails: {\tt\small \{lijiayun22, luyw20\}@mails.tsinghua.edu.cn, ylmo@tsinghua.edu.cn}.}%
}

\begin{document}

\maketitle
\thispagestyle{empty}
\pagestyle{empty}

\begin{abstract}

This paper considers the output prediction problem for an unknown Linear Time-Invariant (LTI) system. In particular, we focus our attention on the OBF-ARX filter, whose transfer function is a linear combination of Orthogonal Basis Functions (OBFs), with the coefficients determined by solving a least-squares regression. We prove that the OBF-ARX filter is an accurate approximation of the Kalman Filter (KF) by quantifying its online performance. Specifically, we analyze the average regret between the OBF-ARX filter and the KF, proving that the average regret over $N$ time steps converges to the asymptotic bias at the speed of $O(N^{-0.5+\epsilon})$ almost surely for all $\epsilon>0$. Then, we establish an upper bound on the asymptotic bias, demonstrating that it decreases exponentially with the number of OBF bases, and the decreasing rate $\tau(\boldsymbol{\lambda}, \boldsymbol{\mu})$ explicitly depends on the poles of both the KF and the OBF. Numerical results on diffusion processes validate the derived bounds.

\end{abstract}

\section{Introduction}
The problem of modeling unknown systems has long been a focal point of extensive study within the control community. Recent interest in data-driven methods has spurred a significant amount of research into adaptive modeling algorithms, such as online output prediction algorithms for unknown systems and the quantification of their online performance~\cite{oymak_revisiting_2022, simchowitz_learning_2019,Tsiamis2023}. The output prediction algorithms aim to minimize the output prediction error, commonly measured by the Mean Square Error (MSE), based on historical inputs and outputs~\cite{tufa_closed-loop_2011}. Such algorithms play a crucial role in controller design, such as in model-predictive control problems~\cite{muddu_reparametrized_2010} and direct data-driven control methods~\cite{breschi2023data}, as well as offering substantial benefits in practical scenarios, such as GPS navigation~\cite{Tsiamis2023}. However, the task of output prediction for an unknown partially observed LTI systems is challenging due to the problem's nonlinearity and nonconvexity~\cite{yu_identification_2018}.

It is widely established that given the true system parameters and the noise statistics, the Kalman Filter (KF) is the optimal output predictor for LTI systems among all linear predictors~\cite{humpherys2012fresh}. As a result, for unknown LTI systems, several recent works focus on identifying the systems' KF using historical samples and quantifying the online performance of the algorithms, for instance, the algorithms' regret over time~\cite{Tsiamis2023,zhang2023learning}. Moreover, several other works are concerned with the online performance of the linear-quadratic-Gaussian regulator, where the identification of KF is also included~\cite{al2023sample,ziemann2022regret}.

On the other hand, another line of work leverages the well-studied Orthogonal Basis Function (OBF) methods to parametrize the online prediction filter in the frequency domain. These OBF-based filters are shown to require fewer parameters than the conventional FIR, ARX and ARMAX filters and to mitigate parameter inconsistency issues encountered in ARX and ARMAX filters~\cite{tufa2012system}. Specifically, this type of methods rely on OBF, a classic framework in system identification. The OBF method proposes to approximate the transfer function of an LTI system using the linear combination of a group of OBFs, such as the Laguerre bases, the Kautz bases, and the Generalized OBF (GOBF), with the poles of bases selected based on \emph{a priori} system knowledge. With a set of predetermined OBFs, the identification problem reduces to a least-squares regression. Notably, Van den Hof et al.~\cite{van_den_hof_system_1995} and Heuberger et al.~\cite{heuberger_generalized_1995} quantify the asymptotic bias of the GOBF, revealing that the bias decreases exponentially w.r.t. the number of OBF bases. They also prove that the solutions to the least-squares regression converge with probability~1~\cite{van_den_hof_system_2005}.

Subsequently, various researches propose online output predictors based on the OBF. One type of methods introduce the OBF-ARX filter by parametrizing the ARX filter using the OBF bases~\cite{van_den_hof_system_2005, PATWARDHAN2005819, BOUZRARA2012848}, hence reducing the problem to least-squares regression. The resulting filter is proved to converge in $L^2$ to the asymptotic solution of the regression for stable SISO systems~\cite[Chapter 4]{van_den_hof_system_2005}. Madakyaru et al.~\cite{muddu_reparametrized_2010} further introduce the state-space form of the OBF-ARX filter for MIMO systems by splitting the filter into multiple multi-input single-output filters. Additionally, Tufa et al.~\cite{tufa2012system} propose to parametrize the rational function before the input in ARX and ARMAX structures with OBF, introducing the ARX-OBF and OBF-ARMAX methods, respectively. However, to the best of our knowledge, the connection between these OBF-based filters and the KF has not been extensively studied, which naturally raises a research question:

\emph{Is the OBF-based filter an accurate approximation to the KF with online performance guarantees?}

In this paper, we study the online average regret between the OBF-ARX filter and the KF. The steady-state KF is considered as an LTI system, with past inputs and outputs as its inputs and the prediction value as its output. As a result, we approximate the transfer function of KF using a linear combination of OBF bases, resulting in an analogous form as the OBF-ARX filter for SISO systems~\cite[Chapter 4]{van_den_hof_system_2005}. Moreover, we quantify the average regret $\mathcal R_N$ of the filter over $N$ time steps, demonstrating that $\mathcal R_N\sim \bar\alpha\tau(\boldsymbol{\lambda}, \boldsymbol{\mu})^{-(\#\text{bases})}+O(N^{-0.5+\epsilon})$ almost surely for a constant $\bar\alpha>0$ and any small number $\epsilon>0$. The bound shows that the average regret converges at the rate of $O(N^{-0.5+\epsilon})$ to the asymptotic bias almost surely, where the bias can be exponentially reduced by increasing the number of the OBF bases. Additionally, the exponential decreasing rate $\tau(\boldsymbol{\lambda}, \boldsymbol{\mu})$ explicitly depends on the pole locations $\boldsymbol{\lambda}$ of the KF and those  locations $\boldsymbol{\mu}$ of the GOBF. This result also demonstrates that \emph{a priori} knowledge of the KF can facilitate the pole selection of the OBF-ARX filter. The main contribution of the paper is threefold:
\begin{itemize}
    \item We study the relationship between the OBF-ARX filter and the KF, showing that OBF-ARX is an accurate approximation of the KF.
    \item We prove that the online average regret of the OBF-ARX filter over $N$ time steps converges to the asymptotic bias at the speed of $O(N^{-0.5+\epsilon})$ almost surely.
    \item We derive an upper bound on the asymptotic bias of the average regret between the OBF-ARX filter and KF, demonstrating that the bias deceases exponentially w.r.t. the number of bases.
\end{itemize}
Numerical examples on diffusion processes are provided to validate the theoretical results.

\subsubsection*{Paper Structure}
Section~\ref{sec:preliminary} briefly introduces the OBF bases and the OBF-ARX filter in the literature. In Section~\ref{sec:main_results}, we conduct an analysis on the average regret of the OBF-ARX filter. Section~\ref{sec:simulations} leverages numerical results to validate the derived bounds. Finally, Section~\ref{sec:conclusions} concludes the paper. Proofs for all the theorems are provided in the appendix.

\subsubsection*{Notations}
$\mathbb{R}$ is the set of all real numbers. $\mathbb{C}$ is the set of all complex numbers. $\mathbb{R}^{a \times b}$ is the set of $a \times b$ real matrices, and $\mathbb{C}^{a \times b}$ is the set of $a \times b$ complex matrices. $I_n$ denotes the $n$-dimensional unit matrix. $\mathbf{0}_n$ denotes an $n$-dimensional all-zero column vector and $\boldsymbol{0}_{m\times n}$ represents an $m\times n$ zero matrix. $\boldsymbol{1}_n$ and $\boldsymbol{1}_{m\times n}$ are similarly defined. $A^\top$ denotes the matrix transpose and $A^H$ denotes the conjugate transpose of $A$. $\|A\|$ denotes the $2$-norm for a vector $A$ or a matrix $A$, and $\mathcal{H}_2$ norm for a system. For Hermitian matrices $A\in\mathbb{C}^{n\times n}$, $A\succeq 0$ implies that $A$ is positive semi-definite. We say that $|f(k)|\sim O(g(k))$ for $g(k)>0$ if there exists a constant $M>0$, such that $\lim_{k\to\infty}\frac{|f(k)|}{g(k)}\leq M$ for all $k=1, 2, \cdots$. The extended expectation $\Ss$ is defined as $\bar{\mathbb E}(\phi(t))=\lim_{t\to\infty}\mathbb E(\phi(t))$ for a proper $\phi$.

\section{Preliminary}\label{sec:preliminary}
The Orthogonal Basis Function (OBF) method tackles the identification problem for stable LTI systems. Consider an $n$-dimensional stable LTI system with $p$ inputs and $m$ outputs:
\begin{equation}
    y(t)=G(z)u(t)+v(t),
\end{equation}
where $G\in\mathcal H_2$, $u$ is a quasistationary signal and $v$ is a martingale difference sequence. $z$ denotes the time-shift operator, i.e., $zu(t)=u(t+1)$. The objective of the OBF method is to identify the system's transfer function $G(z)$ by approximating it through a linear combination of a set of predefined OBFs, $V_k(z)$:
\begin{equation}
    \check G(z)=\sum_{k=1}^{\check q} L_kV_k(z).
\end{equation}
The OBFs $V_k(z)$ can be selected as either scalar functions or, given additional information regarding each instance of the transfer function, as matrix-valued functions~\cite{ninness_obf_mimo}. In this paper, for ease of discussion, we focus on scalar OBFs. Consequently, the coefficients of linear combination are complex matrices $L_k\in\mathbb C^{p\times m}$.
Common examples of OBF include Laguerre bases, Kautz bases, and Generalized OBF (GOBF)~\cite{van_den_hof_system_1995}. See~\cite{van_den_hof_system_2005} for more details. Specifically, both the Laguerre and Kautz bases are special cases of the GOBF~\cite{van_den_hof_system_1995}:
\begin{align}\label{eq:gobf}
    V_{j+(k-1)n_b}(z)&=e_j^\top (zI-A_b)^{-1}B_b[G_b(z)]^{k-1}, \nonumber\\
    & k=1, 2, \cdots, j=1, \cdots, n_b,
\end{align}
where the inner function $G_b(z)$ denotes an $n_b$-th order all-pass filter with poles $\mu_1, \cdots, \mu_{n_b}$, $e_j$ is the $j$-th canonical vector and $(A_b, B_b, C_b, D_b)$ denotes a minimal balanced realization of $G_b(z)$~\cite{van_den_hof_system_1995}.
%

In the OBF method, leveraging \emph{a priori} information of the system, the system operator first selects the order $n_b$, the pole locations $\mu_1, \cdots, \mu_{n_b}$ of the inner function $G_b(z)$, and the number of the bases $\check q$, ensuring that the resulting GOBFs $V_1(z), \cdots, V_{\check q}(z), V_{\check q+1}(z), \cdots$ are complete over the $\mathcal H_2$ space~\cite{van_den_hof_system_2005}. These parameters subsequently allow for the computation of the GOBF $V_1(z), \cdots, V_{\check q}(z)$ in accordance with~\eqref{eq:gobf}. The primary objective then becomes the following least-squares problem:\footnote{Notice that both the closed-loop system and the OBF bases are stable, the extended expectation is well-defined.}
\begin{equation}
    \min_{L_1, \cdots, L_{\check q}}\Ss\left\|y(t)-\sum_{k=1}^{\check q}L_kV_k(z)u(t)\right\|^2,
\end{equation}
using input and output samples $\{u(t), y(t)\}$ of the system $G(z)$. Specifically, given $N$ sample pairs $u(1\!:\!N)=\{u(1), \cdots, u(N)\}, y(1\!:\!N)=\{y(1), \cdots, y(N)\}$, the resulting identification algorithm solves the following least-squares regression:
\begin{equation}\label{eq:finite_least_squares_1}
    \min_{L_1, \cdots, L_{\check q}}\frac{1}{N}\sum_{t=1}^N \left\|y(t)-\sum_{k=1}^{\check q} L_kV_k(z)u(t)\right\|^2,
\end{equation}
with its optimal solution denoted as $L_1(N), \cdots, L_{\check q}(N)$,
and the identified system model becomes:
\begin{equation}
    \check G(z)=\sum_{k=1}^{\check q} L_k(N)V_k(z).
\end{equation}
The least-squares problem~\eqref{eq:finite_least_squares_1} can be further derived into a recursive formulation~\cite{van_den_hof_system_2005}. Additionally, the asymptotic identification bias for SISO systems using the GOBF is quantified~\cite{van_den_hof_system_1995}, which is summarized as follows. Consider the identification of the SISO system $G(z)$ with $n$ poles $\lambda_1, \cdots, \lambda_n$, using $\check q$ GOBF bases $V_1(z), \cdots, V_{\check q}(z)$ with the inner function $G_b(z)$, whose poles are chosen as $\mu_1, \cdots, \mu_{n_b}$, and it is assumed without loss of generality that $\check q=qn_b$. Moreover, denote the asymptotic estimate as
\begin{align}
    &\begin{bmatrix}
        (L_1^*)^\top & \cdots & (L_{\check q}^*)^\top
    \end{bmatrix}^\top=\Ss(\phi(t)\phi(t)^H)^{-1}\Ss(\phi(t)y(t)^H),\nonumber \\
    &\phi(t)=\begin{bmatrix}
        (V_1(z)u(t))^\top & \cdots & (V_{\check q}(z)u(t))^\top
    \end{bmatrix}^\top.
\end{align}
The following theorem provides an upper bound on the asymptotic approximation bias between $G$ and $\check G$.
\begin{theorem}[Asymptotic bias of GOBF~\cite{van_den_hof_system_1995}]\label{thm:gobf_bias}
    For a small constant $\delta>0$, let
    \begin{equation}\label{eq:tau_definition}
        \tau(\boldsymbol{\lambda}, \boldsymbol{\mu})\triangleq \left(\max_{j=1, \cdots, n}\prod_{k=1}^{n_b}\left|\frac{\lambda_j-\mu_k}{1-\bar\lambda_j\mu_k}\right|\right)^{2/n_b}+\delta.
    \end{equation}
    Then, there exists a constant $c>0$, such that the system approximation bias satisfies
    \begin{equation}
        \begin{aligned}
            \left\|\sum_{k=1}^{\check q} L_k^*V_k(z)-G(z)\right\|^2 \leq \tilde c\frac{\tau(\boldsymbol{\lambda}, \boldsymbol{\mu})^{(q+1)n_b}}{1-\tau(\boldsymbol{\lambda}, \boldsymbol{\mu})^{n_b}},
        \end{aligned}
    \end{equation}
    where $\|\cdot\|$ denotes the system $\mathcal H_2$ norm, $\tilde c=c\left(1+\frac{\mathrm{ess}\sup_{\omega}\Phi_u(\omega)}{\mathrm{ess}\inf_{\omega}\Phi_u(\omega)}\right)^2$ and $\Phi_u(\omega)$ denotes the power spectral density of the input signal $u(t)$.
\end{theorem}
Additionally, parametrizing the ARX filter via OBF methods leads to the OBF-ARX filter for SISO systems introduced in~\cite[Chapter 4]{van_den_hof_system_2005}:
\begin{equation}\label{eq:obf_arx}
    \check y(t)=\sum_{k=1}^{\check q_u}L_k^u V_k^u(z)u(t)+\sum_{k=1}^{\check q_y}L_k^yV_k^y(z)y(t),
\end{equation}
where $V_k^u(z)$ and $V_k^y(z)$ can be selected as two different sets of OBFs, and the coefficients $L_k^u, L_k^y$ are obtained by solving a least-squares problem given $N$ samples $\{u(t), y(t)\}_{t=1}^N$:
\begin{equation}
    \min_{L_k^u, L_k^y}\frac{1}{N}\sum_{t=1}^N\left\|y(t)-\check y(t)\right\|^2.
\end{equation}
The least-squares problem can be further derived into a recursive manner~\cite[Chapter 4]{van_den_hof_system_2005}.

\section{Main Results}\label{sec:main_results}
To facilitate the regret analysis, we first derive an OBF-ARX filter by approximating the transfer function of KF using the linear combination of OBFs and describe the corresponding online prediction algorithm. Subsequently, we analyze the online average regret of the derived OBF-ARX filter.

\subsection{Problem Formulation}
Consider an $n$-dimensional LTI system with $m$ inputs and $p$ outputs taking the following form:
\begin{align}\label{eq:true_system}
    x(t+1)&=Ax(t)+Bu(t)+w(t),\nonumber \\
    y(t)&=Cx(t)+v(t),
\end{align}
where $w(t)$ and $v(t)$ represent i.i.d. zero mean random noise with covariances $Q\succeq 0$ and $R\succ 0$, respectively. $w(t)$ and $v(t)$ are mutually independent.
\begin{assumption}\label{assump:convergence}
    The system is controlled by a stabilizing linear control law, formulated by an $n_u$-dimensional LTI system:
        \begin{align}\label{eq:stabilizing_controller}
            \psi(t+1)&=A_u \psi(t)+B_u y(t)+w_u(t),\nonumber \\
            u(t)&=C_u \psi_u(t)+v_u(t),
        \end{align}
        where $w_u(t)$ and $v_u(t)$ are i.i.d. random noise with zero mean and covariances $Q_u\succeq 0$ and $R_u\succeq 0$, respectively. $w_u(t)$ and $v_u(t)$ are mutually independent.\label{assump:control_input}
\end{assumption}
\begin{remark}
    The controller described in~\eqref{eq:stabilizing_controller} encompasses typical control inputs, for instance, the white noise inputs commonly used in system identification and linear controllers such as PID.
\end{remark}

For simplicity of the analysis, we assume that the system has reached the steady state after operating for a sufficiently long time. Now we formulate the output prediction problem~\cite{Tsiamis2023}, which aims to use historical inputs and outputs to minimize the output prediction error in the minimum MSE sense. In this paper, we focus on linear predictors, where the predicted output is a linear combination of past inputs and outputs. The problem is formulated as:
\begin{equation}\label{eq:filter_optimal}
    \hat y^*(t\mid t-1)=\arg\min_{y\in\mathcal L_t}\mathbb E[\|y(t)-y\|^2],
\end{equation}
where $\mathcal L_t$ denotes the set of all possible linear combinations of past inputs $u(-\infty:t-1)$ and outputs $y(-\infty:t-1)$.
\begin{definition}[Average regret]
    The average regret over $N$ time steps for an online linear output predictor $\check y(t\mid 1:t-1)$ using finite past information $u(1:t-1), y(1:t-1)$ is defined as: 
    \begin{equation}\label{eq:regret}
        \mathcal R_N=\frac{1}{N}\sum_{t=1}^N \|\check y(t\mid 1:t-1)-\hat y^*(t\mid t-1)\|^2.
    \end{equation}
\end{definition}

It is well established that the steady-state KF provides an optimal solution to~\eqref{eq:filter_optimal} given the true system parameters $A, B, C$ and covariance matrices $Q, R$~\cite{humpherys2012fresh}:
\begin{equation}\label{eq:KF}
    \begin{aligned}
        \hat x(t+1)&=A(I-KC)\hat x(t)+Bu(t)+AKy(t), \\
        \hat y^*(t)&=C\hat x(t),
    \end{aligned}
\end{equation}
where $K=PC^H(CPC^H+R)^{-1}$ and $P$ is the solution to the algebraic Riccati equation:
\[P=APA^H-APC^H(CPC^H+R)^{-1}CPA^H+Q.\]
Hence, the average regret of the filter $\check y$ can be reformulated as follows leveraging the optimality of the steady-state KF among all linear predictors:
\begin{equation}
    \mathcal R_N=\frac{1}{N}\sum_{t=1}^N \|\check y(t)-\hat y^*(t)\|^2,
\end{equation}
where $\hat y^*(t)$ is the predicted output of the steady-state KF using $u(-\infty\!:\!t-1), y(-\infty\!:\!t-1)$ and we write $\check y(t\mid 1\!:\!t-1)$ as $\check y(t)$ for simplicity of notations.

\subsection{OBF-ARX Filter Inspired by KF}
For the output prediction problem without relying on explicit system parameters, we derive an OBF-ARX filter inspired by KF. By viewing the KF as a stable LTI system that takes inputs $u(t), y(t)$ to produce the estimated output $\hat y^*(t)$, we can approximate the transfer function of KF $\hat G(z)$ using the linear combination of a set of scalar GOBFs $V_1(z), \cdots, V_{\check q}(z)$ with its inner function as the $n_b$-th order all-pass filter $G_b$ and the number of bases $\check q=qn_b$ without loss of generality. This formulation leads to a variant of the OBF-ARX filter:
\begin{equation}\label{eq:OBF-ARX_form}
    \begin{aligned}
        \check y(t)&=\left(\sum_{k=1}^{\check q} L_kV_k(z)\right)\begin{bmatrix} u(t) \\ y(t)\end{bmatrix} \\
        &=\sum_{k=1}^{\check q} L_k^uV_k(z)u(t)+\sum_{k=1}^{\check q} L_k^yV_k(z)y(t),
    \end{aligned}
\end{equation}
where $L_k=[L_k^u\ L_k^y], L_k^u\in\mathbb C^{p\times m}, L_k^y\in\mathbb C^{p\times p}, V_k(z)u(t)=[V_k(z)u(t)_1, \cdots, V_k(z)u(t)_p]$ with $u(t)_k$ denoting the $k$-th element of $u(t)$.
Let $L=[L_1\ \cdots\ L_{\check q}]$,
\begin{equation}\label{eq:check_x_def}
    \check x_k(t)=\begin{bmatrix} V_k(z)u(t) \\ V_k(z)y(t) \end{bmatrix}, \check x(t)=\begin{bmatrix} \check x_1(t)^\top & \cdots & \check x_{\check q}(t)^\top \end{bmatrix}^\top.
\end{equation}

To construct the online prediction algorithm, after selecting the OBFs $V_k(z)$ in a manner similar to the procedure in system identification, we aim to solve the following least-squares problem to optimize the coefficients $L$ using online samples:\footnote{Notice that the closed-loop true system, the system's KF, and the derived OBF-ARX filter, being stable LTI systems, each converges to their unique steady states over time. As a result, the following limits all exist: $\lim_{t\to\infty} \mathbb E[\phi(t)\phi(t)^H], \phi\in\{x, y, \hat y^*, \check x, \check y\}.$}
\begin{equation}\label{eq:KF_least_squares}
    L^*=\arg\min_{L}\Ss \|\hat y^*(t)-L\check x(t)\|^2.
\end{equation}
Since KF provides the optimal output prediction among all linear estimators, it follows that
\begin{align}
    L^*&=\arg\min_{L}\{\Ss\|y(t)-\hat y^*(t)\|^2+\Ss\|\hat y^*(t)-L\check x(t)\|^2\nonumber \\
    &\qquad\qquad\qquad+2\Ss[(y(t)-\hat y^*(t))^\top(\hat y^*(t)-L\check x(t))]\}\nonumber \\
    &=\arg\min_{L}\ \Ss\|y(t)-L\check x(t)\|^2.\label{eq:y_least_squares}
\end{align}
Hence, we can determine $L$ by solving the latter least-squares problem, which has a direct relationship with input and output samples. The online prediction algorithm using the OBF-ARX filter is summarized as:
\begin{enumerate}
    \item Update the state of the filter for $k=1, \cdots, \check q$ with initial states $\check x(0)=\boldsymbol{0}_{\check q(p+m)}, \mathcal W(0)=\boldsymbol{0}_{(p+\check q(p+m))\times (p+\check q(p+m))}$:\footnote{The update of $\check x$ can be further written into a recursive form using the state-space realization of the OBFs~\cite{van_den_hof_system_2005}. The explicit form is omitted due to the space limit.}
    \begin{equation}\label{eq:finite_state_update}
        \begin{aligned}
            \check x_k(t)&=\begin{bmatrix} (V_k(z)u(t))^\top & (V_k(z)y(t))^\top\end{bmatrix}^\top,
        \end{aligned}
    \end{equation}
    \[\checkW(t+1)=\frac{t}{t+1}\checkW(t)+\frac{1}{t+1}\begin{bmatrix}y(t) \\ \check x(t)\end{bmatrix}\begin{bmatrix}y(t) \\ \check x(t)\end{bmatrix}^H, \]
    and predict the output $\check y(t)=L(t)\check x(t)$.
    Since $V_k(z)$ is casual, $\check x(t)$ only possesses information of $y(1\!:\!t-1)$ and $u(1\!:\!t-1)$.
    \item With $L(0)=\boldsymbol{0}_{p\times \check q(p+m)}$, update $L(t)$ by solving a least-squares problem if $\checkW(t+1)$ is invertible. Otherwise, $L(t)=L(t-1)$:\footnote{The least-squares algorithm can be further written into a recursive least-squares form, which is omitted due to the space limit.}
    \begin{equation}\label{eq:finite_least_squares}
        \begin{aligned}
            L(t+1)&=\text{argmin}_{L}\begin{bmatrix}
                I_p & -L
            \end{bmatrix}\checkW(t+1)\begin{bmatrix}
                I_p & -L
            \end{bmatrix}^H.
        \end{aligned}
    \end{equation}
\end{enumerate}
\begin{remark}
    The derived OBF-ARX filter coincides with the predictor introduced in~\cite{vau_closed-loop_2021}, despite being obtained through an entirely different approach.  In subsequent discussions, we shall show that the intrinsic connection between~\eqref{eq:OBF-ARX_form} and the KF facilitates the regret analysis.
\end{remark}

\subsection{Average Regret Analysis}
In this subsection, we quantify the average regret of the derived OBF-ARX filter.

We first introduce another assumption that is common in the OBF literature~\cite{van_den_hof_system_2005}:
\begin{assumption}[Persistent excitation]\label{assump:input}
    Assuming that the input $u(t)$ is persistently exciting of a sufficiently high order, combined with the previous assumption that $R\succ 0$, guarantees that $\Ss(\check x(t)\check x(t)^H)$ and $\checkW(N)$ in the algorithm are invertible for sufficiently large $N$. Additionally, there exists a constant $\underline \alpha>0$, such that $\Ss(\check x(t)\check x(t)^H)\geq \underline\alpha I_m$.\label{assump:convergence_input_excite}
\end{assumption}

Let $\check y^*(t)$ denote the predicted output of the asymptotic OBF-ARX filter, expressed as:
\begin{equation}
    \check y^*(t)=\sum_{k=1}^{\check q} L_k^*\check x_k(t),
\end{equation}
where $L^*=[L_1^*\ \cdots\ L_{\check q}^*]$ is the solution to the following least-squares problem:
\begin{equation}\label{eq:L_star_def}
    L^*=\arg\min_{L} \begin{bmatrix} I_p & -L\end{bmatrix}\Ss\left(\begin{bmatrix} y(t) \\ \check x(t)\end{bmatrix}\begin{bmatrix} y(t) \\ \check x(t)\end{bmatrix}^H\right)\begin{bmatrix} I_p \\ -L^H\end{bmatrix}.
\end{equation}

Then, the average regret~\eqref{eq:regret} is upper bounded by
\begin{equation}\label{eq:regret_decomposition}
    \begin{aligned}
        \mathcal R_N\leq\frac{2}{N}\sum_{t=1}^N\|\check y(t)-\check y^*(t)\|^2+\frac{2}{N}\sum_{t=1}^N\|\check y^*(t)-\hat y^*(t)\|^2.
    \end{aligned}
\end{equation}

The following theorem demonstrates that the first term in~\eqref{eq:regret_decomposition}, quantifying the gap between the identified filter and the asymptotic filter, converges to $0$ almost surely, and the second term converges to the asymptotic bias of the filter. The proof of the theorem is provided in Appendix~\ref{append:convergence}. 
\begin{theorem}[Almost sure convergence]\label{thm:convergence}
    Under Assumption~\ref{assump:convergence} and~\ref{assump:input}, the squared error between the regressed filter and the asymptotic OBF-ARX filter satisfies
    \begin{equation}\label{eq:gap_convergence}
        \lim_{N\to\infty}\frac{\frac{1}{N}\sum_{t=1}^N\|\check y(t)-\check y^*(t)\|^2}{N^{-1+\epsilon}}=0 \ a.s.,
    \end{equation}
    for all $\epsilon>0$. Moreover, the squared error averaged over time between the prediction of KF and that of the asymptotic OBF-ARX filter converges to the asymptotic bias with
    \begin{equation}\label{eq:bias_convergence}
        \lim_{N\to\infty}\frac{\frac{1}{N}\sum_{t=1}^N\|e^*(t)\|^2-\Ss\|e^*(t)\|^2}{N^{-0.5+\epsilon}}=0 \ a.s.,
    \end{equation}
    for all $\epsilon>0$, where $e^*(t)=\check y^*(t)-\hat y^*(t)$. 
\end{theorem}
It remains to quantify the asymptotic bias $\tilde e\triangleq\lim_{t\to\infty}\mathbb E\|\check y^*(t)-\hat y^*(t)\|^2$. Let the transfer function of~\eqref{eq:true_system} be $G(z)=C(zI-A)^{-1}B$. Then, we rewrite~\eqref{eq:true_system} as:
\[y(t)=G(z)u(t)+\epsilon(t),\]
where $\epsilon(t)$ is the system noise composed of $v(t), w(t-1), w(t-2), \cdots$. We now introduce a commonly adopted assumption in the OBF literature~\cite{van_den_hof_system_2005}:
\begin{assumption}\label{assump:noise}
    The system noise sequence $\epsilon(t)$ is a martingale difference sequence.
\end{assumption}
\begin{theorem}[Asymptotic bias]\label{thm:bias_analysis}
        Under Assumption~\ref{assump:convergence} and Assumption~\ref{assump:noise}, there exists a constant $\alpha>0$ independent of $q$ and $n_b$, such that the asymptotic bias between the true system's Kalman filter and the OBF-ARX filter satisfies
        \begin{equation}\label{eq:finite_time_bias}
            \lim_{t\to\infty}\mathbb{E}[\|\check y^*(t)-\hat y^*(t)\|^2]\leq \alpha\frac{\tau(\boldsymbol{\lambda}, \boldsymbol{\mu})^{(q+1)n_b}}{1-\tau(\boldsymbol{\lambda}, \boldsymbol{\mu})^{n_b}},
        \end{equation}
        where the expression of $\tau(\boldsymbol{\lambda}, \boldsymbol{\mu})$ is defined in Theorem~\ref{thm:gobf_bias}, $\boldsymbol{\lambda}=\{\lambda_1, \cdots, \lambda_n\}$ denotes the eigenvalues of the KF and $\boldsymbol{\mu}=\{\mu_1, \cdots, \mu_{n_b}\}$ are the poles of $G_b(z)$ in GOBF.
\end{theorem}
Please refer to Appendix~\ref{append:bias} for the proof of the theorem.
\begin{remark}
    Theorem~\ref{thm:bias_analysis} reveals that the asymptotic bias of the OBF-ARX filter has an exponential decay rate w.r.t. the number of the OBF bases $\check q=qn_b$. Specifically, the decay rate of the asymptotic bias, denoted as $\tau(\boldsymbol{\lambda}, \boldsymbol{\mu})$, relies on the pole locations of OBFs and their relationship with the eigenvalues of the KF. This result indicates that \emph{a prior} knowledge of the KF can facilitate the pole selection of the OBF-ARX filter.
\end{remark}

Finally, we synthesize the results in Theorem~\ref{thm:convergence} and Theorem~\ref{thm:bias_analysis} into the following theorem:
\begin{theorem}[Average regret]\label{thm:regret}
    Under Assumption~\ref{assump:control_input},~\ref{assump:input} and~\ref{assump:noise}, there exists a constant $\bar\alpha>0$, such that the average online regret $\mathcal R_N$ of the OBF-ARX filter over $N$ time steps satisfies
    \begin{equation}
         \mathcal R_N\sim \bar\alpha\tau(\boldsymbol{\lambda}, \boldsymbol{\mu})^{\check q}+O(N^{-0.5+\epsilon}) \ a.s.
    \end{equation}
    for all $\epsilon>0$, where $\check q$ denotes the number of GOBF bases and the expression of $\tau(\boldsymbol{\lambda}, \boldsymbol{\mu})<1$ is defined in Theorem~\ref{thm:gobf_bias}.
\end{theorem}

\section{Simulations}\label{sec:simulations}
This section provides numerical examples to verify the derived bounds.
We consider a heat diffusion process~\cite{mo_network_2009} in a $(3\times 3)\mathrm{m}^2$ square region, with certain obstacles inside, as shown in Fig.~\ref{fig:region_shape}. The two small circles center at $(0.75, 2.25)$ and $(2.25, 2.25)$ respectively, with radius $0.1\mathrm{m}$, and the half circle centers at $(1.5, 0.95)$ with radius $0.55\mathrm{m}$.

We denote the temperature at $(x, y)$ and time $t$ as $s(x, y, t)$, and the dynamics of the diffusion process in the square region can be characterized by the following Partial Differential Equation (PDE):
\begin{equation}
    \frac{\partial s}{\partial t}=\alpha(x, y)\left(\frac{\partial^2 s}{\partial x^2}+\frac{\partial^2 s}{\partial y^2}\right),
\end{equation}
with the boundary condition
\[
        s(x, y, t)=0, \forall (x, y)\in\mathcal{B},
\]
where $\alpha(x, y)$ denotes the diffusion constant at $(x, y)$, and $\mathcal{B}$ denotes the boundary of the region. A heat source in the region is located at $(1.5, 2.25)$, and it is set to the temperature $u_k\sim\mathcal{N}(0, 1)$ at each time step. A sensor is placed at $(1.5, 1.5)$. We then discretize the system in a $10\times 10$ grid into a $100$-dimensional LTI system with the following parameters:
\begin{itemize}
    \item Diffusion constant randomly selected from $\alpha(x, y)\in[0.005, 0.02]$.
    \item Covariance of the noises: $Q=0, R=0.01$.
    \item Parameters of the discretization:
    \begin{itemize}
        \item Time step: $0.1\mathrm{s}$;
        \item Space step: $0.3\mathrm{m}$;
    \end{itemize}
    \item Total simulation time: $200\mathrm{s}$.
\end{itemize}

In our simulations, by injecting i.i.d. Gaussian noise with zero mean and unit covariance as input at each time step and recording the output, we leverage the OBF-ARX algorithm with $10$ Laguerre bases for simulations across $100$ distinct experiments. Each of these experiments is conducted on a different diffusion process, with its diffusion constant being randomly selected from the interval $[0.005, 0.02]$, to ensure a broad representation of system behaviors. The average regret for each system, defined in~\eqref{eq:regret}, is plotted in~Fig.\ref{fig:regret_simulation}, which demonstrates the consistency of the filter's regret with the theoretical bounds in Theorem~\ref{thm:regret}. Moreover, the average asymptotic bias over the $100$ experiments is $9.27\times 10^{-5}$, which is negligible compared to the overall regret.

\begin{figure}
    \centering
    \begin{subfigure}{0.4\linewidth}
        \includegraphics[width=\textwidth]{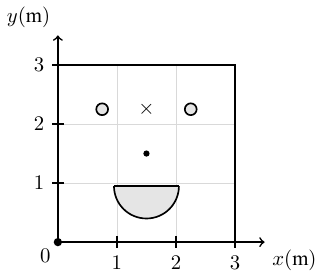}
        \caption{The shape of the region considered in the diffusion process. The obstacles are shown in grey, the heat sources are denoted as $\times$, and the sensor is denoted as the black dot.}
        \label{fig:region_shape}
    \end{subfigure}
    \hspace{0.3cm}
    \begin{subfigure}{0.5\linewidth}
        \includegraphics[width=\linewidth]{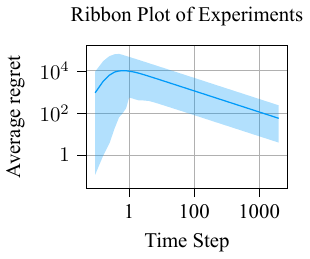}
        \caption{The average regret of the OBF-ARX filter with $10$ Laguerre bases for the heat diffusion process in a log-log ribbon plot. The average asymptotic bias over the $100$ experiments is $9.27\times 10^{-5}$.}
        \label{fig:regret_simulation}
    \end{subfigure}
    \caption{The shape of the region considered in the diffusion process and the average regret.}
\end{figure}

\section{Conclusion}\label{sec:conclusions}

This paper illustrates that the OBF-ARX filter is an accurate approximation of optimal KF in the output prediction problem of an unknown LTI system by quantifying the online average regret of the filter. Specifically, we demonstrate that the average regret of the OBF-ARX filter over $N$ time steps converges to the asymptotic bias at the speed of $O(1/N^{-0.5+\epsilon})$ almost surely for all $\epsilon>0$. Furthermore, we prove that the bias of the average regret decays exponentially w.r.t. the number of OBF bases. Finally, numerical simulations demonstrate that the average regret of the OBF-ARX filter aligns with the theoretical bounds. Notice that the decreasing rate of the asymptotic bias explicitly depends on the relationship between the poles of the KF and those of the OBF, the pole selection of the OBF-ARX filter leveraging \emph{a priori} knowledge of the KF is left for future work.




\appendices

\section{Proof of Theorem~\ref{thm:convergence}}\label{append:convergence}
For simplicity of notations, for a random variable, vector or matrix $x(t)$, we denote that $x(t)\sim \mathcal C_t(\beta)$ if for all $\epsilon>0$, we have that $x(t)\sim O(t^{\beta+\epsilon})\ a.s.$ as $t$ tends to infinity, i.e., $\lim_{t\to\infty}\frac{\|x_t\|}{t^{\beta+\epsilon}}=0$ almost surely. Moreover, define the average function as $\mathcal S_N(\phi(t))=\frac{1}{N}\sum_{t=1}^N \phi(t)$.
We first introduce the following lemmas:
\begin{lemma}[\cite{liu2020online}]\label{lemma:convergence_operation}
    Assuming that $x(t)\sim \mathcal{C}_t(\beta)$ and $y(t)\sim \mathcal{C}_t(\gamma)$, with $\beta\geq \gamma$, then: 1) $x(t)+y(t)\sim\mathcal{C}_t(\beta), x(t)\times y(t)\sim\mathcal{C}_t(\beta+\gamma)$. 2) Suppose $f$ is differentiable at $0$ and $\beta<0$, then $f(x(t))-f(0)\sim \mathcal{C}_t(\beta)$.
\end{lemma}

\begin{lemma}[\cite{liu2020online}]\label{lemma:y_cov_convergence}
    For the noisy state-space model~\eqref{eq:true_system} with $B=\boldsymbol{0}$, we assume that $w(t)$ and $v(t)$ are i.i.d. random signals with zero mean and covariance $Q$ and $R$ respectively, and all $w(t)$ are independent with all $v(t)$. We also assume that the system's initial state $x(0)$ is a zero mean random variable with a finite second moment. Let
    \begin{equation}
        \vartheta(t)\triangleq CA^tx(0)+\sum_{l=0}^{t-1}CA^l w(t-1-l)+v(t),
    \end{equation}
    then
    \begin{equation}
        \frac{1}{N}\sum_{t=1}^N \vartheta(t)\vartheta(t)^H-\mathcal{W}\sim \mathcal{C}_N(-0.5),
    \end{equation}
    with $\mathcal{W}=C\Sigma C^H+R$, where $\Sigma$ is the solution to the Lyapunov equation $\Sigma=A\Sigma A^H+Q$.
\end{lemma}


%
%
Define the covariance matrix
\begin{equation}
    \mathcal{W}(t)=\begin{bmatrix}
        y(t)^\top & \hat x(t)^\top & \check x(t)^\top
    \end{bmatrix}^\top\begin{bmatrix}
        y(t)^H & \hat x(t)^H & \check x(t)^H
    \end{bmatrix},
\end{equation}
where $\hat x(t)$ denotes the state of the KF and $\check x(t)$ is defined in~\eqref{eq:check_x_def}. We first prove the convergence property of $\Sn(\mathcal W(t))$ and $L(N)$.

\begin{theorem}\label{thm:convergence_analysis}
    Under Assumption~\ref{assump:convergence} and~\ref{assump:input}, $\Sn(\mathcal W(t))$ and $L(N)$ satisfies
    \begin{equation}
        \begin{aligned}
        &\Sn(\mathcal W(t))-\Ss(\mathcal W(t))\sim\mathcal C_N(-0.5), \\
        & L(N)-L^*\sim \mathcal C_N(-0.5).
        \end{aligned}
    \end{equation}
\end{theorem}

\begin{proof}
    For simplicity, let us write the dynamics of the closed-loop system with the control input in~\eqref{eq:stabilizing_controller} as the following augmented system:
    \begin{equation}\label{eq:augmented_system}
        \begin{aligned}
            \tilde x(t+1)=\tilde A \tilde x(t)+\tilde w(t),\ \tilde y(t)=\tilde C\tilde x(t)+\tilde v(t),
        \end{aligned}
    \end{equation}
    where
    \[\tilde x(t)=\begin{bmatrix} x(t) \\ \psi(t) \\ u(t) \\ y(t)\end{bmatrix}, \tilde A=\begin{bmatrix}
        A & 0 & B & 0 \\
        0 & A_u & 0 & B_u \\
        0 & C_uA_u & 0 & C_uB_u \\
        CA & 0 & CB & 0
    \end{bmatrix},\]
    \[\tilde y(t)=\begin{bmatrix} u(t) \\ y(t)\end{bmatrix}, \tilde C=\begin{bmatrix} 0 & 0 & I_m & 0 \\ 0 & 0 & 0 & I_p \end{bmatrix},\]
    \[\tilde w(t)=\begin{bmatrix} 
        I & 0 & 0 & 0 \\
        0 & 0 & I & 0 \\
        0 & 0 & C_u & I \\
        C & I & 0 & 0
    \end{bmatrix}\begin{bmatrix}
        w(t) \\ v(t+1) \\ w_u(t) \\ v_u(t+1)
    \end{bmatrix}, \tilde v(t)=\boldsymbol{0}.\]
    It can be verified that $\tilde w(t)$ and $\tilde v(t)$ are mutually independent and each follows an i.i.d. stochastic noise distribution, with their covariances denoted by $\tilde Q$ and $\tilde R$, respectively.
    Let $\tilde n=n+n_u+m+p$ represent the dimension of the augmented system~\eqref{eq:augmented_system}.

    First, we consider the convergence property of $\Sn(\mathcal W(t))$. Denote the state-space realization of the GOBF $V_1(z), \cdots, V_{\check q}(z)$ as the following $\check n_b$-dimensional LTI system:
    \begin{equation}
        \begin{aligned}
            \phi(t+1)&=\check A_b\phi(t)+\check B_{bu}u(t)+\check B_{by}y(t), \\
            \check x(t)&=\check C_b\phi(t).
        \end{aligned}
    \end{equation}
    For simplicity of subsequent derivations, we define the following system that augments the closed-loop system~\eqref{eq:augmented_system}, the KF, and the OBF-ARX filter as:
    \begin{equation}\label{eq:all_augmented_system}
        \begin{aligned}
            X(t+1)=\check AX(t)+W(t), \
            Y(t)=\check CX(t)+V(t),
        \end{aligned}
    \end{equation}
    \[\text{where }X(t)=\begin{bmatrix} \tilde x(t) \\ \hat x(t) \\ \phi(t) \end{bmatrix}, \check A=\begin{bmatrix} \tilde A & & \\
        A_{x\hat x} & A-AKC & \\
        A_{x\check x}  & \boldsymbol{0} & \check A_b
    \end{bmatrix},\]
    \[A_{x\hat x}=\begin{bmatrix}
        \boldsymbol{0}_{n\times n} & \boldsymbol{0}_{n\times n_u} & B & AK\end{bmatrix},
        \]
    \[ A_{x\check x}=\begin{bmatrix}
        \boldsymbol{0}_{\check n_b\times n} & \boldsymbol{0}_{\check n_b\times n_u} & \check B_{bu} & \check B_{by}
    \end{bmatrix},\]
    \[W(t)=\begin{bmatrix} \tilde w(t)^\top & \boldsymbol{0}^\top & \boldsymbol{0}^\top\end{bmatrix}^\top, V(t)=\boldsymbol{0},Y(t)=\begin{bmatrix} y(t) \\ \hat x(t) \\ \check x(t) \end{bmatrix},\]
        \[\check C=\begin{bmatrix}
            \boldsymbol{0}_{(p+n)\times (n+n_u+m)} & I_{p+n} & \boldsymbol{0}_{(p+n)\times \check n_b} \\
            \boldsymbol{0}_{\check n\times (n+n_u+m)} & \boldsymbol{0}_{\check n\times (p+n)} & \check C_b 
        \end{bmatrix},\]
    where $\check n\triangleq \check q(p+m)$ is the dimension of $\check x(t)$. Then, we have that
    \[\Sn(\mathcal W)=\frac{1}{N}\sum_{t=1}^NY(t)Y(t)^H,\]
    with the initial condition of the augmented system:
    \[X(0)=\begin{bmatrix} \tilde x(0)^\top & \hat x(0)^\top & \boldsymbol{0}_{\check n_b}^\top\end{bmatrix}^\top,\]
    where $\tilde x(0)$ and $\hat x(0)$ are already in the steady state. One can verify that the system~\eqref{eq:all_augmented_system} satisfies the assumptions of Lemma~\ref{lemma:y_cov_convergence}. As a result, by Lemma~\ref{lemma:y_cov_convergence},
    \begin{equation}\label{eq:yu_convergence}
        \Sn(\mathcal W(t))-\Ss(\mathcal W(t))\sim \mathcal{C}_N(-0.5).
    \end{equation}

    Next, let us consider the convergence property of $L(N)$. For a Hermitian matrix $X\in\mathbb{C}^{(p+n+\check n)\times (p+n+\check n)}$, define the function $\mathcal{\tilde A}_{l}(X)\triangleq\begin{bmatrix} \boldsymbol{0}_{l\times (p+n+\check n-l)} & I_{l} \end{bmatrix} X \begin{bmatrix} \boldsymbol{0}_{(p+n+\check n-l)\times l} \\ I_{l} \end{bmatrix}$. Moreover, for a matrix $X$ such that $\mathcal{\tilde A}_{\check n}(X)$ is invertible, define a function differentiable at $X$ as:
    \begin{equation}
        \mathcal{A}(X)=\begin{bmatrix} I_{p} & \boldsymbol{0}_{p\times (n+\check n)} \end{bmatrix}X \begin{bmatrix} \boldsymbol{0}_{(p+n)\times \check n} \\ I_{\check n}\end{bmatrix}\mathcal{\tilde A}_{\check n}(X)^{-1}.
    \end{equation}
    According to Assumption~\ref{assump:input}, $\mathcal{\tilde A}_{\check n}(\Ss(\mathcal W(t)))$ is invertible and $\mathcal{\tilde A}_{\check n}(\Sn(\mathcal W(t)))$ is also invertible for a sufficiently large $N$. Therefore, one can verify that
    \[L(N)=\mathcal{A}(\Sn(\mathcal{W}(t))), L^*=\mathcal{A}(\Ss(\mathcal{W}(t))),\]
    for a sufficiently large $N$.
    Then, by Lemma~\ref{lemma:convergence_operation}~2),~\eqref{eq:yu_convergence} and the previous result, we conclude that
    \begin{equation*}\label{eq:Ln_convergence}
        L(N)-L^*\sim \mathcal{C}_N(-0.5).\qedhere
    \end{equation*}
\end{proof}
Combining Theorem~\ref{thm:convergence_analysis} with Lemma~\ref{lemma:convergence_operation}~2), we can quantify the convergence speed of the second term in the average regret~\eqref{eq:regret_decomposition} leveraging the fact that $\tilde A_{n+\check n}(X)$ is differentiable at $\Ss(\mathcal W(t))$:
\begin{equation}\label{eq:regret_term_1}
    \begin{aligned}
        &\Sn(\|\hat y^*(t)-\check y^*(t)\|^2)-\Ss\|\hat y^*(t)-\check y^*(t)\|^2 \\
        &=\begin{bmatrix}C & -L^*\end{bmatrix} \tilde e_N\begin{bmatrix}C^H \\ -(L^*)^H\end{bmatrix}\sim \mathcal C_N(-0.5),
    \end{aligned}
\end{equation}
where $\tilde e_N=(\tilde A_{n+\check n}(\Sn(\mathcal W(t)))-\tilde A_{n+\check n}(\Ss(\mathcal W(t))))$.

Next, we prove a result that directly leads to~\eqref{eq:gap_convergence}. Consider the following stable linear time-variant system alongside its time-invariant counterpart:
\begin{equation}
    \begin{aligned}
        x(t+1)&=Ax(t)+w(t), \\
        y(t)&=C(t)x(t)+v(t),\ y^*(t)=C^*x(t)+v(t),
    \end{aligned}
\end{equation}
where $C(t)$ is a time-variant matrix that depends on past noise signals $w(-\infty:t-1), v(-\infty:t-1)$ and $C(t)-C^*\sim\mathcal C_t(-\beta), -0.5\leq \beta<0$. $w(t)$ and $v(t)$ are i.i.d. random signals with zero mean and covariance $Q$ and $R$ respectively. $w(t)$ and $v(t)$ are mutually independent.

\begin{theorem}\label{thm:ltv_convergence}
    Suppose $C(t)$ satisfies
    \begin{equation}\label{eq:C_convergence}
        C(t)-C^*\sim \mathcal C_t(\beta), -0.5\leq \beta<0,
    \end{equation}
    then
    \begin{equation}
        \frac{1}{N}\sum_{t=1}^N \|y^*(t)-y(t)\|^2\sim \mathcal C_N(2\beta).
    \end{equation}
\end{theorem}
\begin{proof}
    First, 
    \begin{equation}\label{eq:yopt_ycheck_bound}
        \begin{aligned}
            \Sn(\|y(t)-y^*(t)\|^2)\leq \Sn(\|C(t)-C^*\|^2\|x(t)\|^2).
        \end{aligned}
    \end{equation}
    Since $A$ is stable, $\Ss(x(t)x(t)^H)$ is bounded and by Lemma~\ref{lemma:y_cov_convergence}, 
    \[\Sn(x(t)x(t)^H)-\Ss(x(t)x(t)^H)\sim \mathcal C_N(-0.5).\] 
    Hence, there exists a constant $\Phi$, such that $\frac{1}{N}\sum_{t=1}^N\|x(t)\|^2\leq \Phi\ a.s.$ for all $N\geq 1$. On the other hand, by~\eqref{eq:C_convergence}, there exists a constant $\bar C$ independent of $t$, such that for all $0<\epsilon<-\beta$, $\|C(t)-C^*\|\leq \bar Ct^{\beta+\epsilon}$.
    As a result,
    \[\Sn(\|C(t)-C^*\|^2\|x(t)\|^2)\leq \Sn(\bar C^2t^{2\beta+2\epsilon}\|x(t)\|^2)\ a.s.\]
    To bound the Right Hand Side (RHS), we consider the following linear programming problem:
    \begin{align}
        \max_{\xi_t}\ \sum_{t=1}^N \xi_t t^{2\beta+2\epsilon}\ 
        \text{s.t. }&\frac{1}{k}\sum_{t=1}^k \xi_t\leq \Phi, \forall k=1, 2, \cdots, N,\nonumber \\
        &\xi_t\geq 0, t=1, \cdots, N.
    \end{align}
    One can verify that the optimal solution to the problem is $\xi_t=\Phi$. Hence, 
    \begin{equation*}
        \begin{aligned}
            &\Sn(\|C(t)-C^*\|^2\|x(t)\|^2)\leq \Sn(\bar C^2\Phi t^{2\beta+2\epsilon}) \\
            &\leq \frac{\bar C^2\Phi}{N}\left(\int_1^N t^{2\beta+2\epsilon}dt+1\right) \\
            &=\bar C^2\Phi\left(\frac{1}{2\beta+1+2\epsilon}N^{2\beta+2\epsilon}+\frac{2\beta+2\epsilon}{2\beta+1+2\epsilon}N^{-1}\right) a.s.
        \end{aligned}
    \end{equation*}
    Hence, we can conclude that
    \[\Sn(\|C(t)-C^*\|^2\|x(t)\|^2)\sim \mathcal C_N(2\beta).\qedhere\]
\end{proof}
Therefore, the theorem is proved combining~\eqref{eq:regret_decomposition} with~\eqref{eq:regret_term_1} and Theorem~\ref{thm:ltv_convergence}.
\section{Proof of Theorem~\ref{thm:bias_analysis}}\label{append:bias}
We first discuss the upper bound on the power spectral density of the filter's input process $\{u(k), y(k)\}$. We continue to use the notations of the augmented closed-loop system defined in~\eqref{eq:augmented_system}.

\begin{lemma}\label{thm:uy_bound}
    Assuming the closed-loop system~\eqref{eq:augmented_system} is in the steady state, the power spectral density function $\Phi_{\tilde y}(\omega)$ of the sequence $\tilde y(t)$ satisfies
    \begin{equation}\label{eq:Phi_y_bound}
        \mathrm{ess}\sup_{\omega}\|\Phi_{\tilde y}(\omega)\|\leq \frac{1+\rho}{1-\rho}\|\tilde C\|^2\|\tilde\Sigma\|+\|\tilde R\|,
    \end{equation}
    where $\rho<1$ is the spectral radius of $\tilde A$ and $\tilde \Sigma$ is the solution to the Lyapunov equation $\tilde A\tilde \Sigma \tilde A^H+\tilde Q=\tilde \Sigma$.
    Moreover, the covariance matrix $\Ss(\check x(t)\check x(t)^H)$ has the following upper bound with $\tilde p=p+m$:
    \begin{equation}
        \Ss(\check x(t)\check x(t)^H)\leq (\mathrm{ess}\sup_{\omega}\|\Phi_{\tilde y}(\omega)\|)I_{\tilde p\check q}.
    \end{equation}
\end{lemma}
\begin{proof}
    First, we derive the autocorrelation function of $\tilde y(t)$:
    \begin{equation*}
        \Gamma_{\tilde y}(k)\triangleq \mathbb E[\tilde y(t)\tilde y(t+k)^H]=\left\{
            \begin{aligned}
                &\tilde C\tilde \Sigma(\tilde A^H)^k\tilde C^H, k\geq 1, \\
                &\tilde C\tilde A^{-k}\tilde\Sigma\tilde C^H+\tilde R\delta(k), k\leq 0,
            \end{aligned}
        \right.
    \end{equation*}
    where $\delta(k)$ is the Dirac function. Then, the power spectral density function of $\tilde y(t)$ is upper bounded by
    \begin{equation}\label{eq:Phi_y_upper_bound}
        \begin{aligned}
            \|\Phi_{\tilde y}(\omega)\|&\leq 2\sum_{k=1}^\infty \|\Gamma_{\tilde y}(k)\|+\|\Gamma_{\tilde y}(0)\| \\
            &\leq 2\|\tilde C\|^2\|\tilde\Sigma\|\sum_{k=1}^\infty \|A\|^k+ \|\tilde C\|^2\|\tilde \Sigma\|+\|\tilde R\| \\
            &\leq \frac{1+\rho}{1-\rho}\|\tilde C\|^2\|\tilde\Sigma\|+\|\tilde R\|.
        \end{aligned}
    \end{equation}
    As a result,~\eqref{eq:Phi_y_bound} is proved since the RHS of~\eqref{eq:Phi_y_upper_bound} is independent of $\omega$. Moreover, let
    \(V(z)=\begin{bmatrix} V_1(z)I_{\tilde p}\ \cdots\ V_{\check q}(z)I_{\tilde p} \end{bmatrix}^\top,\)
    by the Wiener-Khinchin theorem,
    \begin{equation}\label{eq:frequency_check_x}
        \begin{aligned}
            & \Ss(\check x(t)\check x(t)^H)=\frac{1}{2\pi}\int_{-\pi}^\pi V(e^{i\omega})\Phi_{\tilde y}(\omega)V(e^{i\omega})^H\mathrm{d}\omega \\
            &\leq (\mathrm{ess}\sup_{\omega}\|\Phi_{\tilde y}(\omega)\|)\frac{1}{2\pi}\int_{-\pi}^\pi V(e^{i\omega})V(e^{i\omega})^H\mathrm d\omega \\
            &=(\mathrm{ess}\sup_{\omega}\|\Phi_{\tilde y}(\omega)\|)I_{\tilde p\check q},
        \end{aligned}
    \end{equation}
    where the last step is due to the orthogonality of OBF.
\end{proof}
Then, we can calculate the asymptotic bias of the OBF-ARX filter. For simplicity of derivations, in the rest of the proof, we assume that the OBF-ARX filter has been running for a sufficiently long time such that $\check x(t)$ has reached the steady state. Since $L^*$ is the optimal solution to the problem in~\eqref{eq:KF_least_squares}, the asymptotic bias
\begin{equation}
    \begin{aligned}
        &\mathbb E\|\hat y^*(t)-\check y^*(t)\|^2\leq \mathbb E\|\hat y^*(t)-\tilde L_0\check x(t)\|^2.
    \end{aligned}
\end{equation}
Let the transfer function of the KF from $\tilde y(t)$ to $\hat y^*(t)$ be $\hat G(z)$. Since the OBFs $V_{1}(z), \cdots, V_{\check q}(z), V_{\check q+1}(z), \cdots$ span the $\mathcal H_2$ space~\cite{van_den_hof_system_2005}, there exists $\tilde L_1, \tilde L_2, \cdots\in\mathbb C^{p\times \tilde p}$, such that $\hat G(z)=\sum_{k=1}^{\infty} \tilde L_kV_k(z)$. Let $\tilde L_0=[\tilde L_1\ \cdots\ \tilde L_{\check q}]$, then the output of the steady-state KF is decomposed as:
\[\hat y^*(t)=\tilde L_0\check x(t)+\sum_{k=\check q+1}^\infty \tilde L_k V_k(z)\tilde y(t).\]
Hence, the asymptotic bias is bounded by:
\begin{equation}\label{eq:asymptotic_mse_bound}
    \mathbb E\|\hat y^*(t)-\check y^*(t)\|^2\leq \mathbb E\left\|\sum_{k=\check q+1}^\infty \tilde L_kV_k(z)\tilde y(t)\right\|^2.
\end{equation}
Let $\eta(z)=\sum_{k=\check q+1}^\infty \tilde L_kV_k(z)$, similar to~\eqref{eq:frequency_check_x},
\begin{align}\label{eq:eta_y_bound}
    &\mathbb E\|\eta(z)\tilde y(t)\|^2=\tr\left(\frac{1}{2\pi}\int_{-\pi}^\pi \eta(e^{i\omega})\Phi_{\tilde y}(\omega)\eta(e^{i\omega})^H\mathrm{d}\omega\right)\nonumber \\
    &\leq (\mathrm{ess}\sup_{\omega}\|\Phi_{\tilde y}(\omega)\|)\tr\left(\frac{1}{2\pi}\int_{-\pi}^\pi \eta(e^{i\omega})\eta(e^{i\omega})^H\mathrm{d}\omega\right)\nonumber \\
    &= (\mathrm{ess}\sup_{\omega}\|\Phi_{\tilde y}(\omega)\|)\sum_{k=\check q+1}^\infty \|\tilde L_k\|_{\mathrm F}^2,
\end{align}
where the last step is due to the orthogonality of the OBF and $\|\cdot\|_{\mathrm F}$ denotes the Frobenius norm of a matrix. Next, by adopting the proof of Proposition~6.4 in~\cite{van_den_hof_system_1995} to each element of $[\tilde L_k]_{j, l}$ in the $j$-th row and the $l$-th column of $\tilde L_k$, where $k=\check q+1, \check q+2, \cdots$, it can be similarly demonstrated that a constant $c_{j, l}>0$ exists for $j\!=\!1,\! \cdots\!, p, l\!=\!1\!,\! \cdots\!,\! \tilde p$, such that
\begin{equation}
    \sum_{k=\check q+1}^\infty |[\tilde L_k]_{j, l}|^2\leq c_{j, l}\frac{\tau(\boldsymbol{\lambda}, \boldsymbol{\mu})^{(q+1)n_b}}{1-\tau(\boldsymbol{\lambda}, \boldsymbol{\mu})^{n_b}},
\end{equation}
where the expression of $\tau(\boldsymbol{\lambda}, \boldsymbol{\mu})$ is defined in~\eqref{eq:tau_definition}, with $\boldsymbol{\lambda}=\{\lambda_1, \cdots, \lambda_n\}$ representing the eigenvalues of the KF and $\boldsymbol{\mu}=\{\mu_1, \cdots, \mu_{n_b}\}$ denoting the poles of the GOBF.
The detailed proof is omitted here due to the space limit. Then,
\begin{equation}\label{eq:Le_bound}
    \sum_{k=\check q+1}^\infty\|\tilde L_k\|_{\mathrm{F}}^2\leq c\frac{\tau(\boldsymbol{\lambda}, \boldsymbol{\mu})^{(q+1)n_b}}{1-\tau(\boldsymbol{\lambda}, \boldsymbol{\mu})^{n_b}},
\end{equation}
where $c=\sum_{j=1}^p\sum_{l=1}^{\tilde p}c_{j, l}$.
Hence, combining~\eqref{eq:asymptotic_mse_bound} with~\eqref{eq:eta_y_bound} and~\eqref{eq:Le_bound}, we conclude that there exists $\alpha>0$, such that
\begin{equation}
    \lim_{t\to\infty}\mathbb E\|\check y^*(t)-\hat y^*(t)\|^2\leq \alpha\frac{\tau(\boldsymbol{\lambda}, \boldsymbol{\mu})^{(q+1)n_b}}{1-\tau(\boldsymbol{\lambda}, \boldsymbol{\mu})^{n_b}},
\end{equation}
where $\alpha=c\ \mathrm{ess}\sup_{\omega}\|\Phi_{\tilde y}(\omega)\|$ and $\mathrm{ess}\sup_{\omega}\|\Phi_{\tilde y}(\omega)\|$ is upper bounded by Lemma~\ref{thm:uy_bound}.

\bibliography{ref}
\bibliographystyle{ieeetr}
\end{document}